\documentclass[reqno,english]{amsart}
\usepackage{amsfonts,amsmath,latexsym,verbatim,amscd,mathrsfs,color,array}
\usepackage[colorlinks=true]{hyperref}
\usepackage{amsmath,amssymb,amsthm,amsfonts,graphicx,color}
\usepackage{amssymb}
\usepackage{pdfsync}
\usepackage{epstopdf}
\usepackage{cite}

\newcommand{ \R} { \mathbb{R}}
\usepackage{graphicx}
\usepackage{tikz}

\newtheorem{theorem}{Theorem}[section]
\newtheorem{lemma}[theorem]{Lemma}
\newtheorem{remark}[theorem]{Remark}
\newtheorem{proposition}[theorem]{Proposition}
\newtheorem{corollary}[theorem]{Corollary}
\newtheorem{conjecture}[theorem]{Conjecture}

\numberwithin{equation}{section}

\title[Dancer's conjecture] {On Dancer's conjecture for stable solutions with sign-changing nonlinearity}

\begin{document}
\author[Y. Liu]{Yong Liu}
\address{\noindent Department of Mathematics, University of Science and Technology of China,
Hefei, China, 230026}
\email{yliumath@ustc.edu.cn}

\author[K. Wang]{Kelei Wang}
\address{School of Mathematics and Statistics, Wuhan University, Wuhan, China, 430072}
\email{wangkelei@whu.edu.cn}

\author[J. Wei]{Juncheng Wei}
\address{\noindent Department of Mathematics, University of British Columbia, Vancouver, B.C., Canada, V6T 1Z2}
\email{jcwei@math.ubc.ca}
	
\author[K. Wu]{Ke Wu}
\address{\noindent School of Mathematics and Statistics, Wuhan University, Wuhan, China, 430072}
\email{wukemail@whu.edu.cn}
\begin{abstract}
We establish a Liouville type result for stable solutions for a wide class of second order semilinear elliptic equations in $\mathbb{R}^{n}$ with sign-changing  nonlinearity $f$. Under the hypothesis that the equation does not have any nonconstant one dimensional stable solution, and a further nondegeneracy condition of $f$ at its zero points, we  show that in any dimension,  stable solutions of the equation must be constant.  This partially answers  a question raised by Dancer.
\end{abstract}

\maketitle
\section{Introduction}
We are interested in the classification of stable solutions for general elliptic equation of the form
\begin{equation}\label{eqn}
-\Delta u=f(u),\quad\text{in}\quad\mathbb{R}^{n},
\end{equation}
where the nonlinearity $f$ is assumed to be a $ C^1$ function. Recall that a $C^2$ solution $u$ of this equation is said to be stable,  if and only if
\begin{equation}\label{stability condition}
\int_{\mathbb{R}^{n}}(|\nabla\psi|^{2}- f'(u)\psi^{2})dx\geq 0,\quad\forall\psi\in C_{0}^{\infty}(\mathbb{R}^{n}).
\end{equation}

Stability is an important notion, because many solutions obtained from variational method have finite Morse index, and therefore the classification of stable solutions is useful in the analysis of these solutions. In \cite{Dancer}, Dancer proposed the following stable solution conjecture:
\begin{conjecture}\label{dancer}
Assume $n\leq 8$ and $u$ is a bounded stable solution of the equation \eqref{eqn}.
Then $u$ is one dimensional.
\end{conjecture}
A solution is called one dimensional if it only depends on one direction. Note that  the constant solutions are also one dimensional.  In the Allen-Cahn case, that is, $f(u)=u-u^{3}$, Conjecture \ref{dancer} is closely related to the well-known De Giorgi conjecture concerning the classification of bounded monotone solutions, which are automatically stable:
\begin{conjecture}\label{De-Giorgi}
Let $u$ be a bounded solution of the equation
$$\Delta u+u-u^3=0\quad\text{in}\quad\mathbb{R}^{n}$$
such that $\frac{\partial u}{\partial x_{n}}>0$. Then $u$ is one dimensional, at least if $n\leq 8$.
\end{conjecture}
De Giorgi conjecture was proved in dimension $n=2$ by Ghoussoub and Gui in \cite{Gh-Gui}. For $n=3$, this is proved by Ambrosio and Cabr\'{e} in \cite{Ambrosio}. Savin proved in \cite{Savin}  that for $4\leq n\leq 8$, Conjecture \ref{De-Giorgi} is true under the additional limit condition that
\begin{equation}\label{bea}
u(x_1 ,...,x_n)\rightarrow\pm1\quad\text{as}\quad x_n\rightarrow\pm\infty.
\end{equation}
For $n\geq 9$, counterexamples have been constructed in \cite{del-Michal-Wei}. We emphasize that without any additional assumption like \eqref{bea},  the De Giorgi conjecture is still open for dimensions $4\leq n\leq 8$.

The properties of stable solutions actually have a delicate dependence on the sign of $f$. In the case that  $f$ is nonnegative, in \cite{Cabre-Figalli},  Cabr\`{e}, Figalli, Ros-Oton and Serra  obtained several a priori estimates on stable solutions. Based on these results, they are able
to solve the Brezis's conjecture (see \cite{Brezis}) on the $L^{\infty}$ regularity of ``extremal solutions''. Later on, in \cite{Dupaigne-Farina}, Dupaigne and Farina used the estimates   in \cite{Cabre-Figalli}
 to prove the following Liouville theorem.
\begin{theorem}
 Assume that $u\in C^{2}(\mathbb{R}^{n})$ is  a stable solution of \eqref{eqn} bounded from below, where $f$ is assumed to be nonnegative. If $n\leq 10$, then $u$ must be a constant.
\end{theorem}
In view of Dupaigne and Farina's result, Dancer's conjecture is completely solved when the nonlinearity $f$ is  nonnegative. However, if the nonlinearity $f$ changes sign, to the best of our knowledge,  there are few results on Dancer's conjecture. When the solution is radially symmetric, Cabr\'{e} and Capella \cite{Cabre} proved that if $n\leq 10$, any bounded stable solution of \eqref{eqn}  is a constant solution, regardless of the type of the nonlinearity $f$. In \cite{Dupaigne-Farina1} Dupaigne and Farina established  Liouville type results for some convex nonlinearities. We refer to this paper and its references for more results in this direction. In recent papers \cite{BerestyckiGraham2022,BerestyckiGraham2023}, Berestycki and Graham proved Liouville type and half-space theorems for monostable, ignition and bistable nonlineariries.  Let us also mention that for equations of the special form
	\[-\Delta u+|u|^{q-1}u=|u|^{p-1}u,\]
	where $1\leq q<p$, there are some works on Liouville type results on stable solutions, see \cite{Farina-Yannick},  \cite{Le}, \cite{Selmi}.


In this paper, we will consider Dancer's conjecture for sign-changing nonlinearities. To state our result in a more precise way, let us introduce the following three conditions.
\begin{description}
  \item[{\bf(H1).  Isolated zeros}]

  The zeros of $f$ are isolated, or more precisely,  there exists an integer $m$  such that $f$ vanishes exactly at $m$ points $a_{1}<a_{2}<\cdots < a_{m}$.
  \item[{\bf (H2). Nondegeneracy at stable zeros}]

  There exist  constants $c_{0}, \varepsilon_0>0$ and \begin{equation}
  \label{pupper}
  1\leq p<p^*(n),\end{equation} such that for any $a_j$ satisfying $f^\prime(a_j)\leq 0$, we have
\[c_{0}|t|^{p}\geq-f(a_j+t)sgn(t)\geq c_{0}^{-1}|t|^{p},\quad\text{for}\quad t\in(-\varepsilon_0, \varepsilon_0).\]
Here $sgn$ is the sign function and the exponent $p^*(n)$
is defined by \begin{equation*}
p^*(n)=\left\{\begin{array}{lll}
+\infty,\quad &\text{if}\quad 2\leq n\leq 3,\\
\frac{n(n-2)+4}{n(n-2)-4},\quad &\text{if}\quad 4\leq n.
\end{array}
\right.
\end{equation*}

\item [{\bf (H3).  Non-existence of stable one dimensional profile}]

The ODE
\begin{equation}\label{ODE}
g^{\prime\prime}+f(g)=0\quad\text{in}\quad\mathbb{R}
\end{equation}
does not have any bounded nonconstant stable solution.
\end{description}

Note that monostable and bistable nonlinearies considered in \cite{BerestyckiGraham2022,BerestyckiGraham2023} all satisfy {\bf (H1)}-{\bf (H3)}.

Before proceeding, let us explain these conditions in more details. First of all, unstable zeros of $f$, which by definition are the $a_j$ with $f^\prime(a_j)>0$, are actually not important in our analysis. The reason is that they are unstable constant solutions of \eqref{eqn}.

Secondly, {\bf (H2)} includes the special case when all zeros of $f$ are nondegenerate, i.e. $f^\prime(a_j)\neq 0$ for every $a_j$. In general, {\bf (H2)} can not be removed. This nondegeneracy condition implies that if $a$ is a zero of $f$ such that $f'(a)=0$, then $f$ will change sign around $a$ and has to be negative at the right of $a$ in a small neighbourhood. Now consider the special case $f(u)=u^{p}$, then the equation
$$u''+u^{p}=0,\quad\text{in}\quad\mathbb{R}$$
does not have one dimensional stable solution. However, if $p>p_{\text{JL}}(n)$, then the equation
$$\Delta u+u^{p}=0,\quad\text{in}\quad\mathbb{R}^{n}$$
do have a nonconstant radially symmetric stable solution, where
\begin{equation*}
p_{\text{JL}}(n)=\left\{\begin{array}{lll}
+\infty,\quad &\text{if}\quad 2\leq n\leq 10,\\
\frac{(n-2)^{2}-4n+8\sqrt{n-1}}{(n-2)(n-10)},\quad &\text{if}\quad n\geq 11.
\end{array}
\right.
\end{equation*}
Note that this case does not satisfy the nondegeneracy condition.

On the other hand, we would like to point out that the upper bound on $p$ in (\ref{pupper}) may not be optimal. However, it is still not clear to us what should be the optimal upper bound.

Thirdly, the  Hypothesis {\bf(H3)} will play an important role in our analysis. Roughly speaking, it holds for ``generic'' nonlinearities. To have a better understanding of this condition, we recall that for any solution of \eqref{ODE},  if it is stable, then it is monotone. This fact has already been observed by Dancer.\footnote{This follows by applying Sturm-Liouville comparison theorem to $g^\prime$, which is always a solution to the linearized equation of \eqref{ODE}, that is, an eigenfuction for the quadratic form \eqref{stability condition}  with zero eigenvalue.}  Hence for any bounded stable solution $g$ of \eqref{ODE}, the quantity
\[g_\pm:=\lim_{t\to\pm\infty}g(t)\]
are well-defined. It follows that
\[\frac{1}{2}g^\prime(t)^2\equiv C-F(g(t)),\]
where $F'=f$ and $C$ is a constant determined by $g_\pm$. We then see that
$W(s):=C-F(s)$
is a double-well type potential in the interval $[g_-,g_+]$:
\[0=W(g_\pm)=\min_{s\in[g_-,g_+]}W(s).\]
The above reasoning tells us that if there is a stable solution to \eqref{ODE}, then the equation is essentially of Allen-Cahn type in the corresponding interval. In this sense, our hypothesis exclude the classical Allen-Cahn equation. It is worth pointing out that at this stage, the classification of stable solutions for the Allen-Cahn equation is still open except in dimension $2$.

Now we can state our main result, which is the following
\begin{theorem}\label{main1}
 Under {\bf (H1-H3)}, any bounded stable solution of \eqref{eqn} is a constant.
\end{theorem}

As a direct consequence of Theorem \ref{main1}, we have the following nonexistence result.

\begin{corollary}Suppose $f$ has a unique zero $a$, and $f'(a)>0$. Then  \eqref{eqn} does not have bounded stable solution.
\end{corollary}
Another consequence of Theorem \ref{main1} is the following result on a half space problem.

\begin{corollary}\label{cormain1}
 Suppose the nonlinearity $f$ satisfies {\bf (H1-H3)} on $(0,+\infty)$. If $f(0)\geq 0$ and  $u$ is a bounded positive stable solution of \eqref{eqn} in $\mathbb{R}^{n+1}_{+}:=\{x_{n+1}>0\}$, with
the boundary condition $u=0$ on $\partial\mathbb{R}^{n+1}_{+}$. Then $u$ is one dimensional.
\end{corollary}
\begin{remark}
It is conjectured by Berestycki, Caffarelli and Nirenberg in \cite{Berestycki4} that every bounded positive solution of the above half space problem is one dimensional. For results on the conjecture we refer to \cite{Farina2017,Farina2020} and the references therein.
\end{remark}

The proof of Theorem \ref{main1} uses
the Sternberg-Zumbrun (see \cite{1999On} and \cite{1998A}) inequality, which is an equivalent formulation of the stability condition \eqref{stability condition}.  This inequality requires that for any $\phi\in C_0^\infty(\R^n)$,
\begin{equation}\label{Sternberg-Zumbrun}
\int_{\mathbb{R}^{n}}\left(|\nabla^{2}u|^{2}-|\nabla|\nabla u||^{2}\right)\phi^{2}dx\leq \int_{\mathbb{R}^{n}}|\nabla u|^{2}|\nabla\phi|^{2}dx.
\end{equation}
Here
\[|B(u)|^2:=|\nabla^2u|^2-|\nabla|\nabla u||^2\]
can be regarded as a  curvature term. In fact, if the level set $\{u=u(x)\}$ is smooth near $x$, then
\[|B(u)|^2=|A|^2|\nabla u|^2+|\nabla_T|\nabla u||^2,\]
where $A$ is the second fundamental form of this level set, and $\nabla_T$ denotes the tangential derivative.

In our proof, the Sternberg-Zumbrun inequality will be used in the form of a Caccioppoli inequality. It implies that at most places, $u$ is almost one dimensional (i.e. $|B(u)|^2$ is small). Then we take a cube decomposition of $\R^n$, and divide cubes into bad ones and good ones, depending whether the integral of $|B(u)|^2$ on the cube is small. The integral in the left hand side of \eqref{Sternberg-Zumbrun} can be used to bound the number of bad cubes, while in good cubes, by the nondegeneracy hypothesis {\bf (H2)}, $|\nabla u|$ has a decay away from the bad set. These two estimates are combined to perform an iteration, which eventually leads to a quadratic bound on the energy growth.  We then conclude the proof by a standard log cut-off function technique.

{\bf Notation.}   For every point $\bar{x}=(\bar{x}_{1},\bar{x}_{2},\cdots,\bar{x}_{n})\in\mathbb{R}^{n}$, let
\[Q_{r}(\bar{x})=\left\{x\in\mathbb{R}^{n}:|x_{i}-\bar{x}_{i}|\leq \frac{r}{2},~i=1,2,\cdots, n\right\}\]
be the closed cube centered at $\bar{x}$ with  side length $r$ .

We assume the zeros of $f$ satisfy, for any $j$, $a_{j+1}-a_j\geq 3\varepsilon_0$, where $\varepsilon_0$ is the constant in {\bf (H2)}.

\textbf{Acknowledgement} Y. Liu is supported by  the National Key R\&D Program of China 2022YFA1005400 and NSFC 11971026, NSFC 12141105. K. Wang is supported by  National Key R\&D Program of China (No. 2022YFA1005602) and NSFC (No. 12131017 and No. 12221001). J. Wei is supported by NSERC of Canada.

\section{Proof of Theorem \ref{main1}}
The key in our proof is to use an iteration scheme to establish the quadratic energy growth estimate
$$\int_{Q_R(0)}|\nabla u|^2\leq CR^2.$$
Once this is proved, a standard argument will lead to the desired theorem. Hence in the rest of this section, we assume without loss of generality that $n\geq 3$, since in the dimension two case, the energy automatically has at most quadratic growth in terms of $R$. At this stage, it is worth pointing out that for the one dimensional heteroclinic solution (which is stable) of the Allen-Cahn equation, its energy grows like $O(R^{n-1})$.

We begin with the following
\begin{lemma}\label{lem1}
There exists a positive constant $C$ such that for any $x\in\R^n$ and every $R>1$,
\begin{equation}\label{Sternberg-Zumbrun 2}
\int_{Q_{R}(x)}\left(|\nabla^{2} u|^{2}-|\nabla|\nabla u||^{2}\right)\leq CR^{-2}\int_{Q_{2R}(x)}|\nabla u|^2.
\end{equation}
\end{lemma}
\begin{proof}
  This follows from substituting a standard cut-off function into Sternberg-Zumbrun inequality \eqref{Sternberg-Zumbrun}. More precisely, let $\eta$ be a cutoff function such that $\eta(s)=1$ for $|s|<1$, and $\eta(s)=0$ for $|s|>2$. Then we simply choose $\phi(\cdot):=\eta(|\cdot-x|/R)$ in \eqref{Sternberg-Zumbrun}.
\end{proof}

Since $u$ is bounded, standard regularity theories (see \cite{Gi-Tr}) imply that there exists a constant $C$ such that
\begin{equation}\label{gradient bound}
|\nabla u|\leq C  \quad\text{in}\quad\mathbb{R}^{n}.
\end{equation}
Plugging \eqref{gradient bound} into \eqref{Sternberg-Zumbrun 2}, we get
\begin{equation}\label{first bound on curvature}
  \int_{Q_R(x)}\left(|\nabla^{2} u|^{2}-|\nabla|\nabla u||^{2}\right)\leq CR^{n-2}.
\end{equation}
This bound will be the  starting point of our iteration procedure, where the exponent $n-2$ in the right hand side will eventually be decreased  to a negative one.

\begin{lemma}\label{lem2}
For any $\varepsilon>0$,  there exists a $\delta>0$ such that, for any $x\in\R^n$, if
\begin{equation}\label{pro1eq1}
\int_{Q_1(x)}(|\nabla^{2} u|^{2}-|\nabla|\nabla u||^{2})<\delta,
\end{equation}
then there exists a stable zero $a_j$ such that
 \begin{equation}\label{pro1eq2}
 \|u-a_{j}\|_{L^{\infty}(Q_1(x))}< \varepsilon.
 \end{equation}
\end{lemma}
\begin{proof}
 Assume to the contrary that there is a sequence of points $x^{(k)}$  such that
\begin{equation}\label{pro1eq3}
\int_{Q_1(x^{(k)})}(|\nabla^{2} u|^{2}-|\nabla|\nabla u||^{2})\to 0,
\end{equation}
but for any stable zeros $a_j$,
\begin{equation}\label{pro1eq4}
\|u-a_{j}\|_{L^{\infty}(Q_1(x^{(k)}))}\geq \varepsilon.
\end{equation}
Let $u_{k}(x)=u(x+x^{(k)})$, which is a bounded stable solution of \eqref{eqn} satisfying
\begin{equation}\label{pro1eq5}
\int_{Q_1(0)}(|\nabla^{2} u_{k}|^{2}-|\nabla|\nabla u_{k}||^{2})\to 0.
\end{equation}
Since $\{u_{k}\}$ are uniformly bounded,  combining standard elliptic estimates with  Arzel\`{a}-Ascoli theorem, we know that there exists a function $u_{\infty}\in C^{2}(\R^n)$ such that up to a suitable subsequence, $\{u_{k}\}$ converge to $u_{\infty}$ locally uniformly in $\R^n$.
Passing to the limit in \eqref{pro1eq5} gives
\begin{equation}\label{pro1eq6}
\int_{Q_1(0)}(|\nabla^{2} u_{\infty}|^{2}-|\nabla|\nabla u_{\infty}||^{2})=0.
\end{equation}
Hence $u_\infty$ is one dimensional in $Q_1(0)$. By unique continuation, it is one dimensional in the whole $\R^n$. By {\bf(H3)}, there exists a stable zero $a_j$ such that $u_\infty\equiv a_j$.
As a consequence,
\[\lim_{k\to\infty}\|u-a_{j}\|_{L^{\infty}(Q_1(x^{(k)}))}=\lim_{k\to\infty}\|u_k-a_{j}\|_{L^{\infty}(Q_1(0))}=0.\]
This is a contradiction with \eqref{pro1eq4}.
\end{proof}
Let
\[\mathbb{Z}^{n}=\{(k_{1}, k_{2}, \cdots, k_{n})\in\mathbb{R}^{n}:k_{i}\in\mathbb{Z}, ~i=1,2,\cdots n\}\]
be the integer lattice in $\R^n$. Points in $\mathbb{Z}^n$ are denoted as $\textbf{k}=(k_{1}, k_{2}, \cdots, k_{n})$.  The standard distance function on $\mathbb{Z}^n$ is denoted by $\text{dist}$.

{\bf Construction of good and bad cubes.} Take $\delta_0$ to be the small constant in Lemma \ref{lem2} with $\varepsilon=\varepsilon_0$. Let
\[
\begin{aligned}
\mathcal{B}&=\left\{\textbf{k}\in \mathbb{Z}^n: ~~ \int_{Q_1(\textbf{k})}(|\nabla^{2} u|^{2}-|\nabla|\nabla u||^{2})dx\geq \delta_0\right\},\\
\mathcal{G}&=\left\{\textbf{k}\in \mathbb{Z}^n: ~~ \int_{Q_1(\textbf{k})}(|\nabla^{2} u|^{2}-|\nabla|\nabla u||^{2})dx<\delta_0\right\}.
\end{aligned}
\]
The corresponding bad and good cubes form the sets
\[
\widehat{\mathcal{B}}=\mathop{\cup}\limits_{\textbf{k}\in\mathcal{B}}Q_1(\textbf{k}), \quad
\widehat{\mathcal{G}}=\mathop{\cup}\limits_{\textbf{k}\in\mathcal{G}}Q_1(\textbf{k}).
\]
We first assume that $\mathcal{B}\neq\emptyset$.

For each $R\in\mathbb{N}$, define
\[
\begin{aligned}
N(R)&:=\sharp\left( \mathcal{B}\cap Q_R(0)\right),\\
D(R)&:=\int_{Q_R(0)}\left(|\nabla^2 u|^2-|\nabla|\nabla u||^2\right),\\
E(R)&:=\int_{Q_R(0)}|\nabla u|^2.
\end{aligned}
\]
By Lemma \ref{lem1}, we have the following estimates for these three quantities as the starting point:
\begin{equation}\label{starting point}
	\left\{\begin{aligned}
		& N(R)\leq CR^{n-2},\\
		&D(R)\leq CR^{n-2},\\
		&E(R) \leq CR^n.
	\end{aligned}\right.
\end{equation}

These three quantities are actually intimately related to each other. Indeed, Lemma \ref{lem1} gives
\begin{equation}\label{recursive relation 1}
D(R)\leq CR^{-2}E(2R).
\end{equation}
On the other hand, by a counting of bad cubes, we  have
\begin{equation}\label{recursive relation 2}
	N(R)\leq CD(R).
\end{equation}
To close the loop, we need one more recursive relation to control  $E$ by $N$ . This is the content of the following
\begin{proposition}\label{lem3}Let $p$ be the constant defined in {\bf (H2)}.
	For any $R$ large, $10<L\ll R$, we have
	\begin{equation}\label{recursive relation 3}
		E(R) \leq CL^n N(2R)+CL^{-2\frac{p+1}{p-1}} R^n,
	\end{equation}	where the constant $C$ does not depend on $L$ and $R$.
\end{proposition}

To prove this proposition, we need a decay estimate of $|\nabla u|$ in the good set $\widehat{\mathcal{G}}$. This is the content of the following
\begin{lemma}\label{lem decay in good set}
	There exists a constant $C$ such that for any $\textbf{k}\in \mathcal{G}$, if $\text{dist}\left(\textbf{k}, \mathcal{B}\right)>10$, then
	\begin{equation}\label{decay for gradient}
		\sup_{Q_1(\textbf{k})}|\nabla u| \leq C\text{dist}\left(\textbf{k}, \mathcal{B}\right)^{-\frac{p+1}{p-1}}.
	\end{equation}
\end{lemma}	
\begin{proof}
Let us set $D:=\text{dist}\left(\textbf{k}, \mathcal{B}\right)/2$.  By Lemma \ref{lem2}, for any $x\in Q_D(\textbf{k})$,  there exists a stable zero $a_x$ of $f$ such that
	\[\sup_{Q_1(x)}|u-a_x|\leq \varepsilon_0.\]
	Because $u$ is continuous and zeros of $f$ are isolated,  this zero is the same one for all these $x$. In other words, there exists a stable zero $a$ of $f$ such that
	\[\sup_{Q_D(\textbf{k})}|u-a|\leq \varepsilon_0.\]
	
Now let us define
	\[v(x)=u(x)-a,\quad\text{in}\quad Q_D(\textbf{k}).\]
	For any $x\in Q_D(\textbf{k})$, if $u(x)>a$, then by {\bf (H2)},
	\[\Delta v(x)=\Delta u(x)=-f(v(x)+a)\geq c_{1}|v(x)|^{p}.\]
	In the same way,  if $u(x)<a$, then
	$$-\Delta v(x)=-\Delta u(x)=f(v(x)+a)\geq c_{1}|v(x)|^{p}.$$
	Summing up, we conclude that $v$ satisfies
	\begin{equation}\label{decay equation}
		\Delta|v|\geq c_{1}|v|^{p},\quad\text{in}\quad  Q_D(\textbf{k}).
	\end{equation}
	Then by the Keller-Osserman theory (see \cite{Keller, Osserman})  and a standard blow up argument using the doubling lemma (see \cite{Pe-Pa} ), we can find a constant $C$, which depends only on $p$, such that
	\begin{equation}\label{decay of v}
		\sup_{Q_{D/2}(\textbf{k})}|v|\leq CD^{-\frac{2}{p-1}}.
	\end{equation}
	The estimate \eqref{decay for gradient} then follows from applying standard interior gradient estimate to $u$ in $Q_{D/2}(\textbf{k})$. 
\end{proof}
\begin{remark}
	If $p=1$, \eqref{decay for gradient} should be replaced by
	\begin{equation}\label{decay for gradient, w}
		\sup_{Q_1(\textbf{k})}|\nabla u| \leq Ce^{-c\text{dist}\left(\textbf{k}, \mathcal{B}\right)}.
	\end{equation}
This follows from a similar a priori estimate for the $p=1$ case of \eqref{decay equation}.
\end{remark}

\begin{proof}[Proof of Proposition \ref{lem3}]	
We consider the good and bad set separately.

First of all, let us use
 $\mathcal{B}_L$ to denote the $L$ neighborhood of $\mathcal{B}$ in the lattice $\mathbb{Z}^n$, and $\widehat{\mathcal{B}}_L$ will represent the corresponding domain in $\mathbb{R}^n$. We have
\begin{equation}\label{volume estimate of bad neighborhood}
	\left|\widehat{\mathcal{B}}_L\cap Q_{2R}(0)\right| \leq \sharp \left(\mathcal{B}_L\cap Q_{2R}(0)\right) \leq CL^n  N(2R).
\end{equation}
Then combining \eqref{volume estimate of bad neighborhood} with the gradient bound \eqref{gradient bound}, we get
\begin{equation}\label{energy estimate in bad neighborhod}
	\int_{\widehat{\mathcal{B}}_L\cap Q_{R}(0)}|\nabla u|^2 \leq CL^n N(2R).
\end{equation}

Next, by Lemma \ref{lem decay in good set}, for any $\textbf{k}\notin \mathcal{B}_L$,
\[\sup_{Q_1(\textbf{k})}|\nabla u| \leq CL^{-\frac{p+1}{p-1}}.\]
Therefore
\begin{equation}\label{energy estimate in good}
	\int_{Q_{R}(0)\setminus \widehat{\mathcal{B}}_L }|\nabla u|^2 \leq CL^{-2\frac{p+1}{p-1}} |Q_R(0)| \leq CL^{-2\frac{p+1}{p-1}} R^n.
\end{equation}
Summing up \eqref{energy estimate in bad neighborhod} and \eqref{energy estimate in good}, we obtain \eqref{recursive relation 3}. 
\end{proof}

With all these preparations, we are ready to prove our main result.
\begin{proof}[Proof of Theorem \ref{main1}]
		If $\mathcal{B}=\emptyset$, a direct application of Lemma \ref{lem decay in good set} shows that $u$ is a constant. Hence we assume this is not the case, and we want to establish a quadratic growth bound on $E(R)$ by an iteration with \eqref{starting point} as a starting point.

We will assume without loss of generality that $p>1$. Combining \eqref{recursive relation 1}, \eqref{recursive relation 2} and \eqref{recursive relation 3}, we get
\begin{equation}\label{recursive relation}
	E(R)\leq CL^nR^{-2}E(4R)+CL^{-2\frac{p+1}{p-1}}R^n.
\end{equation}

Let us define  $\alpha_0=0$  and
\begin{equation}\label{iteration}
	\alpha_{k+1}=\left(\alpha_k+2\right)\frac{2(p+1)}{n(p-1)+2(p+1)}.
\end{equation}
Since $p\geq 1$, a simple induction argument tells us that the sequence $\{\alpha_{k}\}$ is increasing and
\begin{equation}
 \label{alp}\alpha_{k}\rightarrow\frac{4(p+1)}{n(p-1)}\quad\text{as}\quad k\rightarrow+\infty.
\end{equation}

We  claim that for any $k=0, \dots$,  there exist $C_k$ such that
\begin{equation}\label{estiE}
E(R)\leq C_kR^{n-\alpha_k}, \quad \forall R>1.\end{equation}
The case of  $k=0$ is just\eqref{starting point}. To see that \eqref{estiE} is indeed true for all $k$, we simply choose $$L=R^{\frac{(2+\alpha_k)(p-1)}{n(p-1)+2(p+1)}}$$ in \eqref{recursive relation} and use an induction argument again.

Now if $p<p^*(n)$, then we can check directly that
\[\frac{4(p+1)}{n(p-1)}>n-2.\]
Therefore, in view of (\ref{alp}), there exists an index $k_0$ such that  $\alpha_{k_0}\geq n-2$. It then follows from (\ref{estiE}) that there exists a $\beta\leq 2$ such that
\[E(R)\leq CR^\beta,  \quad \forall R>1.\]
With this estimate at hand, taking a standard log cut-off function in the Sternberg-Zumbrun inequality (as in \cite{Ambrosio} or \cite{Gh-Gui}), we deduce that  $u$ is one dimensional. But since we have assumed {\bf(H3)}, $u$ must be a constant. 
\end{proof}

With Theorem \ref{main1} at hand, we can prove Corollary \ref{cormain1} in the same spirit as in \cite{Dupaigne-Farina}.
\begin{proof}[Proof of Corollary \ref{cormain1}]
By the same proof of   \cite[Theorem 4]{Dupaigne-Farina}, we know that the function $v:=\lim_{x_n\to +\infty}u$ is a positive bounded stable solution of the equation $$
-\Delta v=f(v),\quad\text{in}\quad\mathbb{R}^{n-1}.
$$ Applying Theorem \ref{main1}, we find that $v$ is a constant. Then still following the proof  of   \cite[Theorem 4]{Dupaigne-Farina}, we conclude that $u$ is one dimensional.
\end{proof}


\bibliographystyle{plain}
\bibliography{lwww-reference}

\end{document}